\documentclass[12pt]{article}
\usepackage{a4,fullpage,epsf,psfrag}
\usepackage{enumerate}
\usepackage{epsfig}
\usepackage{graphicx}
\usepackage{graphics}
\usepackage{latexsym}
\usepackage{epsfig}
\usepackage{amsfonts}
\usepackage{amsthm}
\usepackage{amsmath}
\usepackage{amssymb}
\usepackage[all]{xy}

\usepackage{amscd}
\usepackage{mathrsfs}
\usepackage{stmaryrd}
\usepackage{amsopn}
\usepackage[active]{srcltx}
\usepackage{amscd}
\usepackage{mathrsfs} 
\usepackage{amsopn} 
\usepackage[ps2pdf,breaklinks=true,colorlinks=true]{hyperref}
\usepackage[colorlinks=true]{hyperref}

\newtheorem{theorem}{Theorem}[section]
\newtheorem{proposition}[theorem]{Proposition}

\newtheorem{lemma}[theorem]{Lemma}

\newtheorem{corollary}[theorem]{Corollary}

\newcommand{\R}{{\mathbb R}}
\newcommand{\C}{{\mathbb C}}
\newcommand{\Z}{{\mathbb Z}}

\newcommand{\T}{{\mathbb T}}

\newcommand{\op}[1]{\!\!\mathop{\rm ~#1}\nolimits}

\newenvironment{remark}{\refstepcounter{theorem}\par\medskip\noindent{\bf
Remark~\thetheorem.}}{\unskip\nobreak\hfill\hbox{}\par}

\newenvironment{example}{\refstepcounter{theorem}\par\medskip\noindent{\bf
Example~\thetheorem.}}{\unskip\nobreak\hfill\hbox{}\par}

\newenvironment{definition}{\refstepcounter{theorem}\par\medskip\noindent{\bf
Definition~\thetheorem.}}

\renewcommand{\geq}{\geqslant}
\renewcommand{\leq}{\leqslant}

\begin{document}

\title{{\bf Log\--concavity and symplectic flows}}

\author{Yi Lin and \'Alvaro Pelayo}
\date{}
\maketitle

\begin{abstract}
Let $(M,\omega)$ be a  compact, connected symplectic
$2n$\--dimensional  manifold on which an
$(n-2)$\--dimensional torus $T$ acts effectively and Hamiltonianly. Under the assumption that there is an effective complementary $2$\--torus acting on $M$ with symplectic orbits,
we show that the Duistermaat\--Heckman measure of the $T$\--action is log\--concave. This verifies
the logarithmic concavity conjecture for a class of
inequivalent $T$\--actions. Then we use this conjecture to prove the following:  if
there is an effective symplectic action of an $(n-2)$\--dimensional
torus  $T$ on a compact, connected symplectic $2n$\--dimensional manifold that admits an effective complementary symplectic action of a $2$\--torus with symplectic orbits, then the existence of $T$\--fixed points implies that the $T$\--action is Hamiltonian. As a consequence of this, we give new proofs of a classical theorem by McDuff about $S^1$\--actions, and some of its recent extensions.
\end{abstract}

\section{Introduction}

Let $T$ be a torus, i.e. a compact, connected, commutative Lie group.
This paper proves  the logarithmic
concavity conjecture of the
Duistermaat\--Heckman density function for a class of inequivalent Hamiltonian $T$\--actions
on compact, connected, symplectic manifolds $(M, \omega)$
of dimension $4+2\dim T$. In particular, we address the case of $S^1$\--actions on
$6$\--dimensional manifolds, which has received much interest recently.

Then we
use logarithmic concavity methods
to show that the Hamiltonian character of the symplectic
$T$\--action is equivalent to the fact that there are some
$T$\--fixed points. Our
results imply, with different proofs, a classical theorem of McDuff
about circle (i.e. $1$\--dimensional tori) actions on $4$\--manifolds \cite{Mc1986} (see Corollary \ref{Mc1986})  and its
extension to complexity\--$1$ actions
by Kim \cite{kim} (see Theorem \ref{complexity-one-mcduff}
and Section \ref{remarks}).

In fact, one of the fundamental questions in equivariant symplectic geometry is
to what extent there is a close relationship between:  (i) a
symplectic $T$\--action being Hamiltonian; and (ii)
the symplectic $T$\--action having fixed points. In the late fifties
Frankel \cite{Fr1959} proved  that on compact connected K\"ahler
manifolds, a  $2\pi$\--periodic vector field whose flow preserves the K\"ahler form (equivalently, a K\"ahler $S^1$\--action) is
Hamiltonian if and only if has a fixed point.  This easily implies
that in the compact, connected,
K\"ahler setting, conditions (i) and (ii)
are equivalent for any torus $T$ (see for instance Step 4 in
\cite[Proof of Theorem 3]{PeRa2010}).

 In the early eighties, McDuff proved that if the fixed point set of a symplectic action of a circle on a compact
connected symplectic four manifold is non-empty,  then the action must be Hamiltonian. However, she also constructed
an example of a symplectic circle on a compact, connected symplectic six manifold with non-empty fixed point set which is not Hamiltonian.  So while it is true that if a symplectic $T$\--action
on a compact, connected symplectic manifold is Hamiltonian, then it does have fixed points, the converse does not hold in general in dimension six. However, McDuff's result did not have a transversal toric symmetric, in
which case symplectic flows be Hamiltonian. Precisely, we have the following.

\begin{theorem} \label{a}
Let $(M,\omega)$ be a
compact, connected, symplectic $6$\--dimensional manifold equipped with an effective symplectic
action of a $2$\--torus $S$ on $M$ whose orbits are symplectic. Let $\{\varphi_t\}_{t \in \mathbb{R}}$  be
a $2\pi$\--periodic effective symplectic flow  which commutes with the flows
generated by the $S$\--action. If   $\{\varphi_t\}_{t \in \mathbb{R}}$  has equilibrium points, then it
 is a Hamiltonian flow.
\end{theorem}

\begin{remark}
Theorem \ref{a} can be stated as follows.
Let the circle $S^1$ act
effectively and with non\--empty fixed point set on a
compact, connected, symplectic $6$\--dimensional manifold $(M,\omega)$. Suppose that there is an effective commuting symplectic action of a $2$\--torus $S$ on $M$ whose orbits are symplectic.
Then the action of $S^1$ on $(M,\, \omega)$ is Hamiltonian.
\end{remark}

Theorem \ref{a} is a particular case of the more general Theorem \ref{theorem1} below.

The problem of relating (i) and (ii) is still open in most cases, see Section \ref{remarks}.  In the present paper we study this problem
by studying another problem which is very interesting on its own right, and has generated a lot
of activity: the log\--concavity conjecture for Duistermaat\--Heckman density functions.

In what follows we assume that $n$ is an integer with $n\ge 2$ (though the case $n=2$ is almost always trivial).

\subsection*{Symplectic actions with transversal symplectic symmetry}

Consider the symplectic action of a torus $T$ on a compact, connected symplectic manifold $(M,\,\omega)$ of
dimension $2n$. The integer
$n-\text{dim}\, T$ is called the \emph{complexity of the $T$\--action}. In view of McDuff's example,
if the complexity of a symplectic $T$\--action on $M$ is two,
the action does not have to be Hamiltonian even if the fixed point set if non-empty.

Nevertheless, the following theorem asserts that if the action of $T$ has complexity two,
 and if there is a transversal toral symmetry with symplectic orbits,
then the action of $T$ must be Hamiltonian provided the fixed point set of the action is non-empty.

\begin{theorem} \label{theorem1}
Let $T$ be an $(n-2)$\--dimensional torus which acts
effectively with non\--empty fixed point set on a
compact, connected, symplectic $2n$\--dimensional manifold $(M,\omega)$. Suppose that there is an effective commuting
symplectic action of a $2$\--torus $S$ on $M$ whose orbits are symplectic.\footnote{the following statements are equivalent (\cite[Corollary 2.2.4]{P10}):
let $(M, \,\omega)$ be a compact connected symplectic manifold equipped
with an effective symplectic action of a $2$\--torus $S$: a) \emph{at least one} principal $S$\--orbit is a symplectic submanifold; b)
\emph{every} principal
$S$\--orbit is a symplectic submanifold; c)

\emph{at least one} $S$\--orbit is
a $2$\--dimensional symplectic submanifold; d)
\emph{every} $S$\--orbit is
a $2$\--dimensional symplectic submanifold; e) \emph{every} $S$\--orbit is a symplectic submanifold.}
Then the action of $T$ on $(M,\, \omega)$ is Hamiltonian.
\end{theorem}

We will prove Theorem \ref{theorem1} using logarithmic concavity techniques and the structure
theory of symplectic actions. The proof combines several
tools developed by others, and which we adapt to the setting
of the paper: Lerman's symplectic cutting method \cite{Le95},
Yan's primitive
decomposition of differential forms \cite{Yan}, Graham's work on
logarithmic concavity of push\--forward measures
\cite{Gr96}, the Lerman\--Guillemin\--Sternberg jump formulas
\cite{GLS88}, and Hodge\--Riemann bilinear relations.

\subsection*{Logarithmic concavity of Duistermaat\--Heckman densities}

Our second theorem
is  a proof of the log\--concavity
conjecture of Duistermaat\--Heckman functions of Hamiltonian torus
actions with transversal toral symmetry with symplectic orbits.
We will prove this prior to proving Theorem \ref{theorem1} because
its proof uses this theorem.       In order to state it, consider the
momentum map $\mu \colon M \to \mathfrak{t}^*$ of the Hamiltonian $T$\--action
in Theorem \ref{theorem1}, where $\mathfrak{t}$ is the Lie algebra of $T$, and $\mathfrak{t}^*$ its
dual Lie algebra.

If $V\subset M$ is an open set, the \emph{Liouville measure on $M$} of $V$ is defined as
$$\int_V \frac{\omega^n}{n!}.$$ The
\emph{Duistermaat\--Heckman measure}  on $\frak{t}^*$ is the
push\--forward of the Liouville
measure on $M$ by $\mu$. The Duistermaat-Heckman measure is
absolutely continuous with respect to the Lebesgue measure, and
its density function\footnote{which is well defined once the normalization
of the Lebesgue measure is declared}
$${\rm DH}_{T} \colon \mathfrak{t}^* \to [0,\, \infty)
$$
is  the \emph{Duistermaat\--Heckman function}  (see  Section \ref{pre}).
Recall that if $V$ is a vector space, and if a Borel measurable function
$f: V\rightarrow \R$ is positive almost everywhere, we say that $f$ is \emph{log-concave} if
 its logarithm is a concave function.

\begin{theorem}\label{log-concavity-result}
 Let $T$ be an $(n-2)$\--dimensional torus which acts effectively
 and Hamiltonianly on a
compact, connected, symplectic $2n$\--dimensional
manifold $(M,\omega)$. Suppose that there is an effective commuting
symplectic action of a $2$\--torus $S$ on $M$ whose orbits are symplectic.
Then the Duistermaat-Heckman function ${\rm DH}_{T}$ of the
Hamiltonian $T$\--action is log-concave, i.e. its logarithm is a concave function.
\end{theorem}

The log\--concavity of the Duistermaat-Heckman function was proved by Okounkov \cite{Ok1996,Ok1997} for projective algebraic varieties, and
by Graham \cite{Gr96} for arbitrary compact K\"ahler manifolds. In view of these positive results, it was conjectured independently
by Ginzburg and Knudsen in 1990's that the
Duistermaat\--Heckman function of any Hamiltonian torus action on a compact symplectic manifold is log-concave.

However,
later Karshon \cite{Ka} gave an example of a
Hamiltonian circle action on a compact connected symplectic
$6$\--manifold with a non log\--concave
Duistermaat\--Heckman function, hence disproving the
conjecture in general. More recently, Lin found a general construction
of six dimensional Hamiltonian circle manifolds with non-log-concave Duistermaat\--Heckman functions, and established the log-concavity
conjecture for a class of Hamiltonian torus actions with complexity two \cite{Lin07}.

\subsection*{Applications: $S^1$\--actions and complexity\--$1$ actions}

We obtain, as a consequence of Theorems
\ref{theorem1} and \ref{log-concavity-result},  a theorem of Kim (Theorem \ref{complexity-one-mcduff}) that says  that if $(M,\, \omega)$ is a $2n$\--dimensional compact, connected, symplectic manifold,
then every effective symplectic action of an $(n-1)$\--dimensional torus on $(M,\omega)$ with non-empty fixed point set is Hamiltonian.  If $n=2$, this statement is due to McDuff
\cite{Mc1986} (see Corollary \ref{Mc1986}).  Our
proofs are conceptually different, and use different tools, from that of Kim and McDuff. Theorems
\ref{theorem1} and \ref{log-concavity-result} cover wider classes of manifolds,
see Example \ref{xxxxx}.

\paragraph{{\bf Structure of the paper.}}
The rest of the paper is devoted to preparing the ground for the proofs,
proving these two theorems, and exploring their applications.

In Section \ref{pre} we review the basic elements of symplectic manifolds and torus actions that we need in the remaining of the paper.

We prove Theorem \ref{log-concavity-result} in Section \ref{DH}, and we prove
Theorem \ref{theorem1} in Section \ref{MC} (the proofs of Theorems \ref{theorem1}
and \ref{log-concavity-result} rely mainly on: (1) the
 structure theory of non\--Hamiltonian symplectic torus actions \cite{P10};
 (2) the logarithmic convexity theorem for push\--forward measures \cite{Gr96}).

 In Section \ref{sec5} we explain how
our theorems can be used to derive  proofs of some well\--known results.

 In Section \ref{remarks}
we discuss some related results to these theorems.

\section{Duistermaat\--Heckman theory}  \label{pre}

This section gives a fast review of the necessary background to
read this paper.

\subsection*{Symplectic group actions}

Let $(M,\omega)$ be a symplectic manifold, i.e. the pair consisting of a smooth manifold $M$  and
a symplectic form $\omega$ on $M$ (a non\--degenerate closed $2$\--form on $M$).
Let $T$ be a torus, i.e. a compact, connected, commutative
Lie group  with Lie algebra $\mathfrak{t}$ ($T$ is isomorphic,
as a Lie group, to a finite product of circles $S^1$).

Suppose that $T$ acts on a symplectic manifold $(M,\, \omega)$  symplectically (i.e., by diffeomorphisms which preserve the symplectic form). We denote by $$(t,m) \mapsto t \cdot m$$ the action $T \times M \to M$ of $T$ on $M$.

Any element
$X \in \mathfrak{t}$ generates a vector field $X_M$
on $M$, called the \emph{infinitesimal generator}, given
by $$X_M(m):= \left.\frac{{\rm d}}{{\rm d}t}\right|_{t=0}
{\rm exp}(tX)\cdot m,$$ where ${\rm exp} \colon
\mathfrak{t} \to T$ is the exponential map of Lie theory and
$m \in M$. As usual, we write
$$
\iota_{X_M} \omega : = \omega(X_M, \cdot) \in \Omega^1(M)
$$
for the contraction $1$\--form. The  $T$-action on
$(M,\omega)$  is said to
be \emph{Hamiltonian} if there exists a smooth
invariant map $\mu
\colon M \to \mathfrak{t}^*$, called the \emph{momentum map},
such that for
all $X \in \mathfrak{t}$ we have that
\begin{eqnarray}
\iota_{X_M}
\omega ={\rm d} \langle \mu,
X \rangle, \label{oo}
\end{eqnarray}
where $\left\langle \cdot , \cdot
\right\rangle : \mathfrak{t}^\ast \times \mathfrak{t}
\rightarrow \mathbb{R}$ is the duality pairing.

\begin{remark}
Theorem \ref{theorem1} says that there exists a smooth map
invariant map $\mu \colon M \to \mathfrak{t}^*$ satisfying \textup{(\ref{oo})}, where $\mathfrak{t}^*$ is
the dual Lie algebra of $T^{n-2}$. \end{remark}

\begin{remark}
The existence
of $\mu$ is equivalent to the exactness of
the $1$\--forms $\iota_{X_M}\omega$ for all $X \in
\mathfrak{t}$. The obstruction to the action
to being Hamiltonian lies in the first de Rham cohomology
group ${\rm H}^1_{\rm dR}(M)$ of $M$, i.e. if
${\rm H}^1_{\rm dR}(M)$ is trivial, then any symplectic
$T$\--action on $M$ is Hamiltonian.
\end{remark}

We say that the $T$-action on $M$ \textit{has fixed points} if
\begin{eqnarray} \label{set}
M^T:=\{m \in M \mid t \cdot m = m, \,\, \textup{for all}\,\, t \in T \} \neq \varnothing.
\end{eqnarray}

The $T$\--action is \emph{effective} if the intersection of all stabilizer
subgroups $$T_m:=\{t \in T \,|\,
t \cdot m=m\}, \,\,m \in M,$$ is the trivial group. The $T$\--action
is \emph{free} if $T_m$ is  the trivial group
for all points $m \in N$.
The $T$\--action is \emph{semifree} if it is free on $M\setminus M^T$, where $M^T$ is given by (\ref{set}).
Note that if a $T$\--action is semifree, then it either leaves every point in $M$ fixed, or it is effective.
Finally, the action of $T$  is called quasi\--free if the stabilizer
subgroup  of every point are connected.

\begin{example}
The simplest example of a Hamiltonian torus action is given by the standard symplectic
sphere $S^2$ with the
rotational $S^1$\--action. This action is also
effective and semifree. It is easy to
check that the momentum map for this action is the height function
$\mu(\theta, \,h)= h$.
\end{example}

\subsection*{Atiyah\--Guillemin\--Sternberg convexity}

We will need the following classical result.

\begin{theorem}[Atiyah \cite{A82}, Guillemin and Sternberg
  \cite{GS82}] \label{gs} If a torus $T$ with Lie algebra $\mathfrak{t}$ acts on a
  compact, connected $2n$\--dimensional symplectic manifold $(M,\,
  \omega)$ in a Hamiltonian fashion, then the image $\mu(M)$ under the
  momentum map $\mu \colon M \to \mathfrak{t}^*$ of the action is a
  convex polytope $\Delta \subset \mathfrak{t}^*$.
\end{theorem}

\subsection*{Reduced symplectic forms and Duistermaat\--Heckman formula}

Consider the Hamiltonian action of a torus $T$ on a symplectic manifold $(M, \omega)$. Let $a\in\frak{t}^*$ be a
regular value of the momentum map $\mu : M \rightarrow \frak{t}^*$ of the action. When the action of
$T$ on M is not quasi-free, the quotient $$M_a:= \mu^{-1}(a)/T$$ taken at a
regular value of the momentum map is not a smooth manifold in general.
However,  $M_a$ admits a smooth orbifold structure
in the sense of Satake \cite{Satake}.

Even though smooth orbifolds are not
necessarily smooth manifolds, they carry differential structures such as differential
forms, fiber bundles, etc. In fact, the usual definition of symplectic
structures extends to the orbifold case. In particular, the restriction
of the symplectic form $\omega$  to the fiber $\mu^{-1}(a)$ descends to a symplectic
form $\omega_a$ on the quotient space $M_a$, cf. Weinstein \cite{Weinstein}.
We refer to \cite{ALR} and \cite{MO} for a modern account of basic notions in
orbifold theory, and to \cite[Appendix]{P10} for the basic definitions in the context
of symplectic geometry.

We are ready to state an equivariant version of the celebrated Duistermaat-Heckman theorem.

\begin{theorem}[Duistermaat\--Heckman \cite{DH1982}]\label{DHtheorem}
Consider an effective Hamiltonian action of a
$k$\--dimensional torus $T^k$ on a compact, connected, $2n$\--dimensional symplectic
manifold $(M,\omega)$ with momentum map $\mu : M\rightarrow \frak{t}^*$.
Suppose that there is another symplectic action of $2$\--torus $T^2$ on $M$ that commutes with
the action of $T^k$. Then the following hold:
\begin{itemize}
\item [{\rm (a)}] at a regular value $a\in \mathfrak{t}^*$ of $\mu$, the Duistermaat-Heckman function
is given by
$$ {\rm DH}_{T^k}(a) =\int_{M_a}\dfrac{\omega_a^{n-k}}{(n - k)!},
$$
where $M_a:= \mu^{-1}(a)/T^k$ is the symplectic quotient, $\omega_a$ is the corresponding
reduced symplectic form, and $M_a$ has been given the
orientation of $\omega_a^{n-k}$;

\item[{\rm (b)}] if $a, a_0 \in \frak{t}$ lie in the same connected component of the set of regular
values of the momentum map $\mu$, then there is a $T^2$\--equivariant diffemorphism
$
F \colon M_a \longrightarrow M_{a_0},
$
where
$M_{a_0}:= \mu^{-1}(a_0)/T^k$ and $M_{a}:= \mu^{-1}(a)/T^k$. Furthermore,  using
$F$ to identify $M_a$ with $M_{a_0}$, the reduced symplectic form
on $M_a$ may be identified with
$$
\omega_a = a_0+  \langle c, a - a_0 \rangle,
$$
where $c\in \Omega^2(M, \frak{t^*})$ is a closed $\frak{t}^*$\---valued two\--form representing
the Chern class of the principal torus bundle $\mu^{-1}(a_0)\rightarrow M_{a_0}$.
\end{itemize}
\end{theorem}

\begin{example} \label{paa}
Under the conditions of Theorem \ref{a}, Theorem \ref{log-concavity-result} says that
the Duistermaat\--Heckman function
$$
a \mapsto \frac{1}{2} \int_{M_a} \omega_a^{2},
$$
is log\--concave,  where $M_a$ is the symplectic quotient $\mu^{-1}(a)/S^1$  and $\omega_a$ is the reduced symplectic form inducing the orientation on $M_a$.
\end{example}

\subsection*{Primitive decomposition of differential forms}

Here we assume that $(X,\omega)$ is a $2n$\--dimensional symplectic orbifold.
Let $\Omega^{k}(X)$ be the space of differential forms of degree $k$ on $X$. Let $\Omega(X)$ be the space corresponding to the
collection of all $\Omega^{k}(X)$, for varying $k$'s.

We note that, in analogy with the case of symplectic manifolds, there are three natural operators
$\Omega(X)$  on given as follows:
\[\begin{split} &{\rm L}: \Omega^{k} (X) \rightarrow \Omega^{k+2}(X),\,\,\,\alpha\mapsto \omega_a\wedge \alpha,\\
& \Lambda: \Omega^{k}(X)\rightarrow \Omega^{k-2}(X),\,\,\,\alpha \mapsto \iota_{\pi}\alpha,\\
&{\rm H}: \Omega^{k}(X)\rightarrow \Omega^{k}(X),\,\,\,\alpha\mapsto (n-k)\alpha,\end{split}\]
 where $\pi=\omega^{-1}$ is the Poisson bi-vector induced by the symplectic form $\omega$. These
 operators satisfy the following bracket relations.
 \[ [\Lambda ,{\rm L}]={\rm H}, \,\,[{\rm H},\Lambda]=2\Lambda,\,\,\,[{\rm H},{\rm L}]=-2{\rm L}. \]
 Therefore they define a representation of the Lie algebra ${\rm sl}(2)$ on $\Omega(X)$. A primitive differential form
 in $\Omega(X)$ is by definition a highest weight vector in the $sl_2$-module $\Omega(X)$. Equivalently, for
 any integer $0\leq k\leq n$,  we say that a differential
 $k$-form $\alpha$ on $X$ is primitive if and only if
 \[ \omega^{n-k+1}\wedge \alpha=0.\]

 Although the ${\rm sl}_2$-module $\Omega(X)$ is infinite dimensional, there are only finitely many eigenvalues of ${\rm H}$.
 There is a detailed study of ${\rm sl}_2$-modules of this type in \cite{Yan}. The following result
 is an immediate consequence of \cite[Corollary 2.5]{Yan}.

\begin{lemma} \label{yan's-result}
A differential form
$\alpha \in \Omega^k(X)$ admits a unique primitive decomposition
\begin{equation}\label{lefschetz-decompose-forms} \alpha =\displaystyle \sum_{r\,\,\geq\,\,
\max(k-n, 0)} \dfrac{{\rm L}^r}{r!}\, \beta_{k-2r},\end{equation}
where $\beta_{k-2r}$ is a primitive form of degree $k-2r$.

\end{lemma}

\section{Proof of  Theorem \ref{log-concavity-result}}
\label{DH}

The proof of Theorem \ref{log-concavity-result} combines ingredients coming from
several directions:
reduction theory for Hamiltonian group actions;
structure theory for symplectic group actions with symplectic orbits;
complex Duistermaat\--Heckman theorem;
Lerman\--Guillemin\--Sternberg wall jump formulas;
Hodge\--Riemann bilinear relations.
We will first explain the particular form in which we need to use these ingredients
above, and then how to combine them to prove the theorem.

\subsection*{Toric fibers and reduction ingredients}

We start with the following observations.

\begin{lemma}\label{level-set} Let a torus $T$ with Lie algebra $\mathfrak{t}$ act
on a compact, connected, symplectic manifold $(M, \omega)$ in a Hamiltonian fashion with momentum map
$\mu: M\rightarrow \frak{t}^*$. Suppose that the action of another torus $S$ on
$M$ is symplectic and commutes with the action of $T$.
Then for any $a\in \frak{t^*}$, the action of $S$ preserves the level set $\mu^{-1}(a)$ of the momentum map.
\end{lemma}

\begin{proof} Let $X$ and $Y$ be two vectors in the Lie algebra of $S$ and $T$ respectively, and let
$X_M$ and $Y_M$ be the vector fields on $M$ induced by the infinitesimal action of $X$ and $Y$ respectively.  Let $\mathcal{L}_{X_M}\mu^Y$ denote the Lie derivative
of $$\mu^Y:=\langle \mu,\, Y\rangle \colon M \to \mathbb{R}$$ with respect to $X_M$.
Let $[\cdot,\,\cdot]$ denote the Lie bracket on vector fields.

 It suffices to show that $\mathcal{L}_{X_M}\mu^Y=0$.

 Since the action of $T$ and $S$ commute, we have that
 $$[X_M,Y_M]=0.$$
 By the Cartan identities, we have
 \[\begin{split} 0&=\iota_{[X_M,Y_M]}\omega=\mathcal{L}_{X_M}\iota_{Y_M}\omega-\iota_{Y_M}\mathcal{L}_{X_M}\omega
 \\&= \mathcal{L}_{X_M}\iota_{Y_M}\omega \,\,\,\,\,\,(\text{because the action of $S$ is symplectic.})
  \\&=\mathcal{L}_{X_M}({\rm d}\mu^Y)={\rm d}\mathcal{L}_{X_M}\mu^Y.  \end{split}\]

This proves that $\mathcal{L}_{X_M}\mu^Y$ must be a constant. On the other hand, we have
\[  \mathcal{L}_{X_M}\mu^Y=\iota_{X_M}{\rm d}\mu^Y
=\iota_{X_M}\iota_{Y_M}\omega
=\omega(Y_M,X_M).\]

Since the action of $T$ is Hamiltonian and the underlying manifold $M$ is compact, there must exist a point $m\in M$ such that $Y_M$
vanishes at $m$. Thus $$(\mathcal{L}_{X_M}\mu^Y)(m)=0,$$ and therefore
$\mathcal{L}_{X_M}\mu^Y$ is identically equal to zero on $M$.
\end{proof}

\begin{remark}\label{level-set-remark} It is clear from the above argument that if the action of $T$ is symplectic with a generalized momentum map, cf. Definition
\ref{generalized-momentum-map}, and if the fixed point set of the $T$-action is non-empty, then the assertion of Lemma \ref{level-set} still holds.

\end{remark}

\begin{lemma}\label{T2action-on-quotient}
 Let an $(n-2)$\--dimensional torus $T^{n-2}$ with Lie algebra $\mathfrak{t}$ act effectively
on a compact connected symplectic
manifold $(M, \omega)$ in a Hamiltonian fashion with momentum map
$\mu: M\rightarrow \frak{t}^*$.
Suppose that there is an effective commuting $T^2$\--action on $(M,\omega)$ which has a symplectic orbit.
 Then for any  regular value $a\in \frak{t}^*$, the symplectic quotient
$$M_a:=\mu^{-1}(a)/T^{n-2}$$ inherits an effective symplectic $T^2$\--action with symplectic orbits.
\end{lemma}

\begin{proof} Let $a$ be a regular value of the momentum map $\mu \colon M\rightarrow \frak{t}^*$. Then the symplectic quotient $M_a$ is a four dimensional symplectic orbifold.   By assumption, the $T^2$\--action on $M$ has a symplectic orbit. Hence all the orbits
of the $T^2$\--action on  $M$ are symplectic \cite[Corollary 2.2.4]{P10}.

By Lemma \ref{level-set}, the $T^2$\--action preserves the level set $\mu^{-1}(a)$, and thus it descends to an action on the reduced space $M_a$. It follows that
the orbits of the induced $T^2$\--action on $M_a$ are symplectic two dimensional orbifolds. Thus, by dimensional considerations, any isotropy subgroups of the
$T^2$\--action on $M_a$ is zero\--dimensional. Since
these isotropy subgroups must be closed subgroups of $T^2$, they must be finite subgroups of $T^2$.

Now let $H$ be the intersection of the all isotropy subgroups of $T^2$. Then $H$
is a finite subgroup of $T^2$ itself.
Note that the quotient group $T^2/H$ is a two\--dimensional compact,
connected commutative  Lie group. Hence $T^2/H$ is isomorphic to
$T^2$. Therefore  $M_a$ admits an effective action of $T^2\cong T^2/H$ with symplectic orbits.
\end{proof}

\subsection*{Symplectic $T$\--model of $(M,\omega)$}
Let $(X,\sigma)$ be a connected symplectic $4$\--manifold equipped with an effective
symplectic action of a $2$\--torus $T^2$ for which the $T^2$\--orbits are symplectic.
We give a concise overview of a model of $(X,\sigma)$, cf. Section \ref{appendix}.

Consider
 the quotient map $\pi: X\rightarrow X/T^2$.
Choose a base point $x_0\in X$, and let $p_0:=\pi(x_0)$. For any homotopy class $[\gamma]\in \widetilde{X/T^2}$,
i.e., the homotopy class of a loop $\gamma$ in $X/T^2$ with base $p_0$, denote by
$$\lambda_{\gamma}:[0,1]\rightarrow
X$$ the unique horizontal lift of $\gamma$ ,with respect to the flat connection $\Omega$ of symplectic orthogonal
complements to the tangent spaces to the $T$\--orbits, such that $\gamma(0)=x_0$, cf. \cite[Lemma 3.4.2]{P10}.  Then the map
\[\Phi: \widetilde{X/T^2}\times T^2\rightarrow  X,\,\,\, ([\gamma], t)\mapsto t\cdot\lambda_{\gamma}(1)\]
is a smooth covering map and it induces a
$T^2$\--equivariant symplectomorphism  $$\widetilde{X/T^2} \times _{\pi^{\rm orb}_1(X/T^2)} T^2 \to X.$$
We refer to see Section \ref{appendix}, and in particular Theorem \ref{tt} and the remarks below it, for
a more thorough description.

\subsection*{K\"ahler ingredients}

Here we review how a $T^2$-invariant complex structure is constructed on a connected symplectic
manifold $4$\--manifold $(X,\sigma)$ as above.

\begin{lemma}\label{invariant-kahler-structure} Consider an effective symplectic action of a $2$\--torus $T^2$ on a connected \footnote{not necessarily compact} symplectic $4$\--manifold $(X,\sigma)$ with a symplectic
orbit. Then $X$ admits a $T^2$-invariant K\"ahler structure. It consists of a $T^2$-invariant complex structure, and a
K\"ahler form equal to $\sigma$.
\end{lemma}

\begin{proof}
If $X$ is compact, the lemma is a  particular case of
Duistermaat\--Pelayo \cite[Theorem 1.1]{DP10}. The proof therein is given by constructing the complex
structure on the model $\widetilde{X/T^2} \times _{\pi^{\rm orb}_1(X/T^2)} T^2$ of $M$.
In fact, the same proof therein
applies even if $X$ is not compact. We review it for completeness.

The orbifold universal covering
$\widetilde{X/T^2}$
of the orbisurface $X/T^2$
may be identified with the
Riemann sphere, the Euclidean plane, or the hyperbolic plane,
on which the orbifold fundamental group $\pi^{\rm orb}_1(X/T^2)$
of $X/T^2$ acts by means of orientation preserving isometries, see
Thurston \cite[Section 5.5]{thurston}. Provide $\widetilde{X/T^2}$ with the
standard complex structure.

Let $\sigma^{T^2}$ be the unique $T^2$\--invariant symplectic
form on $T^2$ defined
by $\sigma^{\mathfrak{t}}$ at the beginning of Section \ref{appendix}: the restriction
of $\sigma$ to any $T$\--orbit. Equip $T^2$ with a $T^2$\--invariant
complex structure such that $\sigma^T$ is equal to a
K\"ahler form.

In this way
we obtain a $T^2$\--invariant complex structure
on $\widetilde{X/T^2} \times _{\pi^{\rm orb}_1(X/T^2)} T^2$
with symplectic form $\sigma$ equal to the K\"ahler form.
\end{proof}

Suppose that $a$ is a regular value of $\mu: M\rightarrow \frak{t^*}$. Let $M_a^{\text{reg}}$ be the set of smooth points in $M_a$.
Then it follows from Lemma \ref{invariant-kahler-structure} that $M_a^{\text{reg}}$ is a smooth manifold that admits a $T^2$ invariant K\"ahler structure, which consists of a $T^2$-invariant complex structure and a $T^2$-invariant K\"ahler form $\omega_a$.

 \subsection*{Duistermaat-Heckman theorem in a complex setting}

  The following lemma is a key step in our approach.

\begin{lemma}\label{key-lemma}
Consider an effective Hamiltonian action of an $(n-2)$\--dimensional torus $T^{n-2}$ with Lie algebra $\mathfrak{t}$ on a
compact, connected, $2n$\--dimensional
symplectic manifold $(M,\omega)$ with momentum map
$\mu \colon M \to \mathfrak{t}^*$.
Suppose that there is an effective commuting
symplectic $T^2$\--action on $M$ which has a symplectic orbit, that $a$ and $b$ are two regular values
of $\mu$ which lie in the same connected component of the regular values of the momentum map, and that $$M_a:= \mu^{-1}(a)/T^{n-2}\,\,\, \textup{and}\,\,\, M_b:=\mu^{-1}(b)/T^{n-2}$$ are symplectic quotients
 with reduced symplectic structure $\omega_a$ and $\omega_b$ respectively.
 Let $M_a^{{\rm reg}}$ and $M_b^{{\rm reg}}$ be the complement
 of the sets of orbifold singularities in $M_a$ and $M_b$, respectively.
 Then both $M_a^{{\rm reg}}$ and $M_b^{{\rm reg}}$ are $T^2$\--invariant K\"ahler manifolds as in Lemma {\rm \ref{invariant-kahler-structure}},
 whose K\"ahler forms are $\omega_a$ and $\omega_b$ respectively. Moreover, there is a diffemorphism $F: M_a\rightarrow M_b$ such that
 $F^*\omega_b$ is a $(1,1)$-form on the K\"ahler manifold $M_b^{{\rm reg}}$.

\end{lemma}

\begin{proof}
By Lemma \ref{T2action-on-quotient}, both $(M_a, \omega_a)$ and $(M_b,\omega_b)$ are symplectic four orbifolds that
 admit an effective symplectic $T^2$ action with symplectic orbits.
 The smooth parts $M_a^{\text{reg}}$ and $M_b^{\text{reg}}$ are symplectic manifolds that admit an effective symplectic action of $T^2$ with symplectic orbits.

It follows from Lemma \ref{invariant-kahler-structure} that
 $(M_a^{\text{reg}}, \omega_a)$ and $(M_b^{\text{reg}},\omega_b)$
admit a $T^2$\--invariant complex structure
and the corresponding K\"ahler form is equal to $\omega_a, \omega_b$, respectively.
We prove that $F^*\omega_a$ is a $(1,1)$-form on $M_b^{\text{reg}}$ by writing down
a local expression of $F$ using holomorphic coordinates on $M_a^{\text{reg}}$ and
$M_b^{\text{reg}}$ respectively.

Let $\Omega_a$ the the flat connection given by the
$\omega_a$\--orthogonal complements to the tangent spaces
to the $T^2$\--orbits on $M_a^{\text{reg}}$. Similarly for $\Omega_b$.
Let $\Sigma_a:= M_a^{\text{reg}}/T^2$, $\Sigma_b:= M_b^{\text{reg}}/T^2$,
and let $\widetilde{\Sigma_a}$ and $\widetilde{\Sigma_b}$ be the orbifold
universal covers of $\Sigma_a$ and $\Sigma_b$, respectively, at some base points.

We have two smooth covering maps $$\Phi_a:\widetilde{\Sigma_a}\times T^2\rightarrow M_a^{\text{reg}}$$
and $$\Phi_b:\widetilde{\Sigma_b}\times T^2\rightarrow M_b^{\text{reg}}.$$

Suppose $p_0=t_0\lambda_{\gamma_0}(1) \in M_a^{\text{reg}}$ is an arbitrary point. Choose an open neighborhood $U$ of $[\gamma_0]$ and $V$ of $t_0$ such that the restriction of $\Phi_a$ to $U\times V$ is a diffeomorphism.  By Theorem \ref{DHtheorem}, there
exists a $T^2$-equivariant diffeomorphism
$$F: M_a^{\text{reg}}\rightarrow M_b^{\text{reg}}.$$
 We have that for any $[\gamma]\in U$ and $t\in V$,
\[ F(t\lambda_{\gamma}(1))=t\cdot F(\lambda_{\gamma}(1)).\]

On the other hand, $F$ induces a diffeomorphism between two $2$\--dimensional
orbifolds
\begin{eqnarray} \varphi_{ab} \, \colon\,\Sigma_a \rightarrow
\Sigma_b. \label{map0}
\end{eqnarray}

These orbifolds
can be shown to be good orbifolds \cite[Lemma 3.4.1]{P10}, and hence
their orbifold universal covers are smooth manifolds.
Then we have an induced diffeomorphism between smooth manifolds
\[ \widetilde{\varphi }_{ab}: \widetilde{\Sigma_a} \rightarrow \widetilde{\Sigma_b}.\]

Note that $\widetilde{\varphi }_{ab}(\gamma)$ is a loop  based at $\widetilde{\varphi}_{ab}(p_0)$. Denote by
$\lambda_{\widetilde{\varphi}_{ab}(\gamma)}$ its horizontal lift with respect to $\Omega_b$.
Then we have
\[\lambda_{\widetilde{\varphi}_{ab}(\gamma)}(1)=\tau_{\gamma}\cdot  F(\lambda_{\gamma}(1)) ,\]
for some $\tau_{\gamma} \in T$.

Therefore, using covering maps $\Phi_a$ and $\Phi_b$, the diffemorphism $F: M^{\text{reg}}_a\rightarrow M^{\text{reg}}_b$ has a local expression
\[ ([\gamma], t)\mapsto ([\widetilde{\varphi}_{ab}(\gamma)], \,\tau_{\gamma} t).\]

The proposition follows from the above expression for $F$, and how the complex structures and symplectic structures are constructed on $M_a$ and
$M_b$, cf. Lemma \ref{invariant-kahler-structure}.
\end{proof}

\subsection*{Walls and Lerman\--Guillemin\--Sternberg jump formulas}

Consider the effective Hamiltonian action of a $k$\--dimensional torus
$T^k$ on a $2n$\--dimensional compact symplectic manifold
$(M,\omega)$ with momentum map $\mu: M\rightarrow \frak{t}^*$. By
Theorem \ref{gs}, the image of the momentum map
$\Delta:= \mu(M)$ is a convex
polytope. In fact, $\Delta$ is a union of convex subpolytopes with the property that
the interiors of the subpolytopes are disjoint convex open sets and constitute the set of
regular values of $\mu$.
These convex subpolytopes are called the \emph{chambers of $\Delta$}.

For any circle $S^1$ of $T^k$, and for any connected component $\mathcal{C}$
of the fixed point submanifold of
$S^1$, the image of $\mathcal{C}$ under the momentum map $\mu$ is called
an $(k-1)$\--dimensional wall, or simply a \emph{wall of $\Delta$}.
Moreover, $\mu(\mathcal{C})$ is called an \emph{interior wall of $\Delta$} if
$\mu(\mathcal{C})$ is not a subset of the boundary of $\Delta$.

Now choose a $T^k$-invariant inner product on $\frak{t}^*$ to identify $\frak{t}^*$ with $\frak{t}$.
Suppose that $a$ is a point on a codimension one interior wall $W$ of
$\Delta$,
and that $v$ is a normal vector to $W$ such that the line segment
$$\{a + tv\}$$ is transverse to the
wall $W$. For $t$ in a small open interval near $0$, write
\begin{eqnarray} \label{for:g}
g(t):= {\rm DH}_T(a + tv),
\end{eqnarray} where
${\rm DH}_T$ is the Duistermaat-Heckman function of $T$.

Let $S^1$ be a circle sitting inside $T^k$ generated by $v$.  Let $H$ be a
$(k-1)$\--dimensional torus $H$ such that
$$T^k=S^1\times H.$$ Suppose that $X$ is a connected component of the fixed point submanifold of the $S^1$\--action on $M$ such that
$\mu(X)=W$. Then $X$ is invariant under the action of $H$. Moreover,
the action of $H$ on $X$ is Hamiltonian. In fact, let $\frak{h}$ be the Lie algebra of $H$,
and let $$\pi: \frak{t}^*\rightarrow \frak{h}^*$$ be the canonical projection map. Then the composite
\[\mu_H:=\pi\circ \mu \mid_X: X\rightarrow \frak{h}^*\] is a momentum map for the action of $H$ on $X$.

By assumption,  $a$ is a point in $W=\mu(X)$ such that $\pi(a)$ is a regular value of $\mu_H$. There are
two symplectic quotients associated to $a$: one may reduce $M$ (viewed as an $H$-space) at $\pi(a)$, and one may reduce
$X$ (viewed as a $T^k/S^1$\--space) at $a$. We denote these symplectic quotients by
$M_a$ and $X_{a}$ respectively. It is immediate that $M_{a}$
inherits a Hamiltonian $S^1$\--action and that $X_{a}$ is the fixed point submanifold of the $S^1$\--action on $M_{a}$.

Guillemin\--Lerman\--Sternberg gave the following formula for computing
the jump in the Duistermaat-Heckman function across the wall $\mu(X)$ of
$\Delta$.

\begin{theorem}[Guillemin\--Lerman\--Sternberg \cite{GLS88}] \label{G-L-S88}
The jump of the  function $g(t)$ in expression {\rm (\ref{for:g})} when moving across the interior wall $W$ is given by
\[g_+(t) - g_-(t) \,= \,\displaystyle \sum \textup{volume}(X_{a}) \,\Big(\prod_{i=1}^k\alpha_k^{-1}\Big) \, \dfrac{t^{k-1}}{(k-1)!},\]
plus an error term of order $\mathcal{O}(t^k)$. Here the $\alpha_k$'s are the weights of the
representation of $S^1$ on the normal bundle of $X$,  and the sum is taken over the symplectic quotients of all the connected components
$X$ of $M^{S^1}\cap \mu^{-1}(\alpha)$ with respect to the $T^k/S^1$ action at $a$.
\end{theorem}

Building on Theorem \ref{G-L-S88}, Graham established the following result,
cf. \cite[Section 3]{Gr96}.

\begin{proposition}[Graham \cite{Gr96}] \label{Grahma-concavity}
Suppose that  $\mu: M \rightarrow \mathfrak{t^*}$ is the momentum map of an  effective Hamiltonian action of torus $T$ on a compact, connected symplectic
manifold $(M,\omega)$. Let $a$ be a point in a codimension one interior wall of $\mu(M)$,
and let $v\in\mathfrak{t}^*$ be a vector such that the line segment $\{a + tv\}$ is transverse to the
wall. For $t$ in a small open interval near $0$, write $$g(t) = {\rm DH}_T(a + tv),$$ where
${\rm DH}_T$ is the Duistermaat-Heckman function. Then we have that
$$g_+'(0)\leq g_-'(0).$$
\end{proposition}

\subsection*{Hodge-Riemann bilinear relations and last step in proof of Theorem \ref{log-concavity-result}}

Now we have all the ingredients to conclude the proof of Theorem \ref{log-concavity-result}. Our proof depends essentially on the Hodge-Rieman bilinear relation on a K\"ahler manifold.

Recall that if $V$ is a vector space and and $A\subset V$ is a convex open subset, and if $f  \colon A \to \R$ is a Borel measurable map that is positive on $A$ almost everywhere, then we say that $f$ is \emph{log-concave} on $A$ if and only if $\log  f$ is a concave function on $A$. Moreover, if $f$ is smooth on $A$, and if $a$ is a fixed point in $A$,
then a simple calculation shows that $f$ is log-concave on $A$ if and only if
\[ f''(a+tv)\cdot f(a+tv)-(f')^2(a+tv)\leq 0,\,\,\,\textup{for all}\,\,\,\, v\in A.\]

\begin{proof}[Proof of Theorem \ref{log-concavity-result}] By Proposition \ref{Grahma-concavity}, to
establish the log\--concavity of ${\rm DH}_T$ on
$\mu(M)$, it suffices to show that the restriction of ${\rm log} \,{\rm DH}_T$ to each connected
component of the set of regular values of $\mu$ is concave. Let $\mathcal{C}$ be
such a component, let $v \in \frak{t^*}$, and let $\{a + tv\}$ be a line segment in $\mathcal{C}$
passing through a point $a \in \mathcal{C}$, where the parameter $t$ lies in some
small open interval containing $0$. We need to show that $$g(t) := {\rm DH}_T(a+tv)$$ is
log-concave, or equivalently, that
\begin{equation} \label{inequality} g''g-(g')^2\leq 0 .\end{equation}

Since the point $a$ is arbitrary, it suffices to show that equation (\ref{inequality})
holds at $t=0$.
By Theorem \ref{DHtheorem}, at $a+tv\in \frak{t^*}$, the Duistermaat\--Heckman function is
\begin{equation*}\begin{aligned} {\rm DH}_T(a+tv)=\int_{M_a}\dfrac{1}{2}(\omega_a+tc)^2. \end{aligned}
\end{equation*}
Here $M_a=\mu^{-1}(a)/T$ is the symplectic quotient taken at $a$ and $c$ is a closed two\--form
depending only on $v\in \frak{t^*}$. To prove (\ref{inequality}), we must show that
\begin{equation}\label{key-inequality} \int_{M_a}c^2 \int_{M_a}\omega_a^2 \,\,\leq\,\,  2\left(\int_{M_a}c \,\omega_a\right)^2.\end{equation}

Consider the primitive decomposition of the two\--form $c$ as \[c=\gamma+s\omega_a,\] where $s$ is a real number, and $\gamma$ is a
primitive two\--form. By primitivity,  we have that \[\gamma\wedge \omega_a=0.\] A simple calculation shows that
(\ref{key-inequality}) for $\gamma$ implies (\ref{key-inequality}) for $c$. But then (\ref{key-inequality}) becomes
\begin{equation}\label{key-inequality2} \int_{M_a}\gamma^2 \leq 0.\end{equation}

Note that in a symplectic orbifold the subset of orbifold singularities is of codimension greater than or equal to
$2$, cf. \cite[Prop. III.2.20]{A04}. In particular, it is of measure zero
with respect to the Liouville measure on the symplectic orbifold. Let $M_a^{{\rm reg}}$ be the complement
 of the set of orbifold singularities in $M_a$. It follows that
\[\int_{M_a}\gamma^2 =\int_{M_a^{\text{reg}}}\gamma^2.\]
 Note that $M_a^{\text{reg}}$ is a connected symplectic four\--manifold which admits the effective symplectic action of $T^2$ with symplectic orbits. By Lemma \ref{invariant-kahler-structure},
 $M_a^{\text{reg}}$ is a K\"ahler manifold. Moreover, $\omega_a$ is the K\"ahler form on $M^{{\rm reg}}_a$. In particular, it must be a
 $(1,1)$-form.

 Let $\langle \cdot,\cdot\rangle $ be the metric on the space of $(1,1)$\--forms induced by the K\"ahler metric, and let $*$ be the Hodge star operator induced by the K\"ahler metric. By Lemma \ref{key-lemma}, $c$ must be a $(1,1)$-form. As a result, $\gamma$ must be a real primitive $(1,1)$-forms. Applying Weil's identity, cf. \cite[Thm. 3.16]{W},  we get that
 \[ \gamma=-*\gamma.\] Consequently, we have that
 \begin{eqnarray}
  \int_{M_a^{\text{reg}}}\gamma^2&=&-\int_{M_a^{\text{reg}}}\gamma \wedge *\gamma \nonumber \\
  &=&-\int_{M_a^{\text{reg}}}
 \langle \gamma,\gamma \rangle \,\,
 \leq \,\,0. \nonumber
 \end{eqnarray}
This completes the proof of Theorem \ref{log-concavity-result}.

\end{proof}

\subsection*{Log-concavity conjecture for complexity one Hamiltonian torus actions}

We note that our proof of Theorem \ref{log-concavity-result} and Theorem \ref{theorem1} depend on the classification of symplectic four manifolds with a symplectic $T^2$ action \cite{P10}. However, in the following special case, we have an elementary proof which makes no uses of the classification results therein.

\begin{theorem} \label{complexity-one} Consider
an effective Hamiltonian action of an $(n-2)$\--dimensional torus $T^{n-2}$ on a
 compact, connected, $(2n-2)$\--dimensional symplectic manifold $(M,\omega)$.
 Let $T^2$ be a $2$\--dimensional torus
equipped with the standard symplectic structure.  We consider the following action of
$T^{n-2}$ on the $2n$\--dimensional product symplectic manifold $M\times T^2$,
\begin{equation} \label{product-action} g\cdot (m,t)= (g\cdot m, t), \,\,\,\textup{for all}\,\,\, g\in T^{n-2},\,m\in M,\, t\in T^2.\end{equation} Then the
Duistermaat-Heckman function ${\rm DH}_{T^{n-2}}$ of the $T^{n-2}$ action on $M\times T^2$ is log\--concave.
\end{theorem}

\begin{proof}
Note that there is an obvious commuting
$T^2$ symplectic action on $M$
\begin{equation} \label{product-action2} h\cdot (m,t)= ( m, h\cdot t), \,\,\,\textup{for all}\,\,\, h\in T^{2},\,m\in M,\, t\in T^2,\end{equation}
which is effective and which has symplectic orbits.
Let $\mu$ be the momentum map of the $T^{n-2}$ action on $M\times T^2$, and ${\rm DH}_{T^{n-2}}$ the the Duistermaat-Heckman function.

To prove that ${\rm DH}_{T^{n-2}}$  is log\--concave, in view of  Proposition \ref{Grahma-concavity},
it suffices to show that the restriction of $\ln {\rm DH}_{T^{n-2}}$ to each connected
component of the set of regular values of $\mu$ is concave. Using the same notation as in the proof of
Theorem \ref{log-concavity-result}, we only need to establish the validity of  inequality (\ref{inequality}).
However, a calculation shows that in the current case the function $g$ must be a degree
one polynomial, and so the inequality (\ref{inequality}) holds automatically. This completes the proof of Theorem \ref{complexity-one}.
\end{proof}

The following corollary is an immediate consequence of Theorem \ref{complexity-one}.

\begin{corollary}Suppose that there is an effective Hamiltonian action of an $(n-1)$\--dimensional torus $T^{n-1}$ on a compact, connected, $2n$\--dimensional symplectic manifold $(M,\omega)$. Then the Duistermaat-Heckman function ${\rm DH}_T$ of the $T$\--action is log-concave.
\end{corollary}

\section{Proof of Theorem \ref{theorem1}} \label{MC}

We will use Theorem \ref{log-concavity-result} and several other tools (which we explain next), in order to prove Theorem \ref{theorem1}.

\subsection*{Logarithmic concavity of torus valued functions}

We first extend the notion of log\--concavity to functions defined
on a $k$\--dimensional torus $T^k \cong \ \R^k/\Z^k$.  Consider the covering map
\[ \exp : \R^k\rightarrow T^k,\,\,\,\,(t_1,\cdots, t_k)\mapsto ({\rm e}^{{\rm i}2\pi t_1},\cdots,{\rm e}^{{\rm i}2\pi t_k}).\]

\begin{definition}
We say that a map $f: T^k\rightarrow (0, \infty)$  is \emph{log-concave at a point $ p\in T^k$}
if there is a point $x\in \R^k$ with $\exp(x)=p$ and an open set $V\subset \R^k$ such that
$$
\exp\mid_V: V\rightarrow
\exp(V)$$ is a diffeomorphism and such that the logarithm of
$$\tilde{f}:= f\circ \exp$$ is a concave function
on $V$. We say that $f$ is \emph{log-concave on $T^k$} if it is log-concave at every point of $T^k$.
\end{definition}

The log\--concavity of the function $f$ does not depend on the choice of $V$. Indeed, suppose that there are two points $x_1, x_2\in \R^k$ such that $$\exp x_1=\exp x_2=p,$$ and
two open sets $x_1\in V_1\subset \R^k$ and $x_2\in V_2\subset \R^k$ such that both $$\exp\mid_{V_1}: V_1\rightarrow \exp(V_1)$$
and $$\exp\mid_{V_2}: V_2\rightarrow \exp(V_2)$$ are diffeomorphisms. Set $$\tilde{f}_i=f\circ \exp\mid_{V_i},\,\,\, i=1,2.$$
Then there exists $(n_1,\cdots n_k)\in \Z^k$ such that
\[\tilde{f}_1(t_1,\cdots t_k)=\tilde{f}_2(t_1+n_1,\cdots,t_k+n_k), \,\,\,\textup{for all}\,\, \in (t_1,\cdots, t_k)\in V_1.\]

Now let $$T^k\cong \R/\Z\oplus \R/\Z\oplus \cdots \oplus \R/\Z$$ be a $k$ dimensional torus,
and $\lambda$ the standard length form on $S^1 \cong \R/\Z$. Throughout this section we denote
by $\Theta$ the $\frak{t}^*$\--valued form
\begin{equation}\label{canonical-form}\lambda\oplus \lambda\oplus \cdots \oplus \lambda \end{equation}
on $T^k$. Clearly, $\Theta$ is invariant under the $T^k$\--action on $T^k$ by multiplication.

Now consider the action of a $k$\--dimensional torus $T^{k}$
on a $2n$\--dimensional symplectic manifold $(M,\omega)$.
Let $\frak{t}$ be the Lie algebra of $T^k$, let
$\frak{t}^*$ be the dual of $\frak{t}$, and let $\Theta$
be the canonical $T^{k}$\--invariant $\frak{t}^*$\--valued one form
given as in (\ref{canonical-form}).

\begin{definition}\label{generalized-momentum-map}
  We say that $$\Phi: M\rightarrow T^{k}$$ is a \emph{generalized momentum map} if
\[ \iota_{X_M}\omega =\langle X,\Phi^*\Theta \rangle,\]
where $X\in \frak{t}$, and $X_M$ is the vector field on $M$ generated
by the infinitesimal action of $X$ on $M$.
\end{definition}

\subsection*{Duistermaat\--Heckman densities of Torus valued momentum maps }

The following fact is a straightforward generalization of a well-known result due to D. McDuff concerning the existence of circle valued momentum maps.

\begin{theorem}[McDuff \cite{Mc1986}] \label{m}
Let a $k$\--dimensional torus $T^k \cong \mathbb{R}^k/\mathbb{Z}^k$ act
symplectically on the compact symplectic
manifold $(M, \omega)$. Assume that the symplectic form
$\omega$ represents an integral cohomology class in
${\rm H}^2(M,\Z)$. Then:
\begin{itemize}
\item[{\rm (1)}]
 either the action admits a standard $\R^k$\--valued
momentum map or, if not,
\item[{\rm (2)}]
there
exists a $T^k$\--invariant
symplectic form $\omega'$ on $M$   that
admits a $T^k$\--valued momentum map
$\Phi \colon M \rightarrow T^k$.
\end{itemize}
\end{theorem}

When $\omega$ is integral and $k=1$,  the $1$\--forms $\iota_{X}  \omega$ is also integral
and the map in Theorem \ref{m} is defined as follows.
Pick a point  $m_0 \in M$, let $\gamma_m$ be an arbitrary
smooth path connecting $m_0$ to $m$ in $M$,
and define the map $\Phi:M \rightarrow S^1$ by
\begin{eqnarray} \label{mom}
\Phi(m):= \left[ \int_{\gamma _m}\iota_{X}  \omega\right] \in \mathbb{R}/\mathbb{Z}.
\end{eqnarray}
One can check that $\Phi$ is well\--defined.

\begin{remark}
When $k=1$, a detailed proof of Theorem \ref{m} may be found in Pelayo\--Ratiu \cite{PeRa2010}.
The argument given in \cite{PeRa2010} extends to actions of higher dimensional tori, as pointed
out therein. As noted in \cite{Mc1986}, the usual symplectic quotient construction carries through
for generalized momentum maps.
\end{remark}

Now we will need the following result, which is Theorem 3.5 in  \cite{R}.

\begin{theorem}[Rochon \cite{R}] \label{ro}
 Let a compact connected Lie group $G$  act on a symplectic manifold $(M,\omega)$
 symplectically. Then the $G$\--action on $(M,\omega)$ is Hamiltonian if and only if
 there exists a symplectic form $\sigma$ on $M$ such that the $G$\--action on $(M,\sigma)$ is Hamiltonian.
 \end{theorem}

It follows from Theorem \ref{ro} that in order to show that a
symplectic $S^1$-action on a symplectic manifold $(M,\omega)$ is Hamiltonian,
it suffices to show it under the assumption that the cohomology class $[\omega]$ is
integral, see \cite[Remark 2.1]{KKK}.

We are ready to introduce the following notion.

\begin{definition} Suppose that there is an effective symplectic action of a $k$\--dimensional torus $T^k$
on a $2n$\--dimensional compact symplectic manifold $(M,\omega)$ with a generalized momentum
map $\Phi: M\rightarrow T^k$. The \emph{Duistermaat-Heckman measure} is the push-froward of the
Liouville measure on $M$ by $\Phi$.
Its Density function is the \emph{Duistermaat-Heckman function}.

\end{definition}

Analogous to the case of Hamiltonian actions, we have the following result.

\begin{proposition}Suppose that there is an effective symplectic action of a $k$\--dimensional torus $T^k$
on a $2n$\--dimensional  compact symplectic manifold $(M,\omega)$ with a generalized momentum
map $\Phi: M\rightarrow T^k$. Let $(T^k)_{\rm reg} \subset T^k$ be the set of regular values of $\Phi$.
Then we have that

\[ {\rm DH}_{T^k}(a)= \int_{M_a} \omega^{n-k}_a ,\,\,\,\textup{for all}\,\, a\in (T^k)_{\rm reg}, \]
where $M_a:=\mu^{-1}(a)/T^k$ is the symplectic quotient taken at $a$,
and $\omega_a$ is the reduced symplectic form.

\end{proposition}

\subsection*{Symplectic cuttings}
We also need the construction of symplectic cutting which was first introduced by Lerman \cite{Le95}.
Next we present a review of this construction.

      Let $\mu : M\rightarrow  \frak{t^*}$ be a momentum map for
an effective action of a $k$\--dimensional torus $T^k$ on a symplectic manifold $(M,\omega)$ and let
$\frak{t}_{\mathbb{Z}} \subset \frak{t}$ denote the integral
lattice.

Choose $k$ vectors $v_j \in \frak{t}_{\mathbb{Z}}$, $1\leq j\leq k$. The form
$$\omega-{\rm i}\displaystyle \sum_j
{\rm d}z_j\wedge {\rm d}\overline{z}_j$$
 is a symplectic form on the manifold $M\times \mathbb{C}^k$. The map
 $$\nu : M\times \mathbb{C}^k \rightarrow \R^k$$ with $j$-th component
\[\nu_j(m, z) =\langle \nu(m), \,v_j\rangle -\vert z_j \vert^2\]
is a momentum map for the $T^k$\--action on $M\times \mathbb{C}^k$.

 For any $b=(b_1,b_2,\cdots, b_k) \in \R^k$, consider a convex rational polyhedral
set
\begin{equation} \label{cutting}
P = \{x \in \frak{t}^*\,\, \vert\,\, \langle x, v_i \rangle \geq b_i,\,\,\forall\, 1 \leq i\leq k\}.\end{equation}
The \emph{symplectic cut of $M$ with respect to $P$} is the reduction of
$M\times \mathbb{C}^k$ at $b$. We denote it by $M_P$ .

  Similarly, for any two regular values $b=(b_1,\cdots, b_k), c=(c_1,\cdots, c_k)\in \R^k$ with
  $b_i\leq c_i$ for all
  $1\leq i\leq k$, we define
   a convex rational polyhedral set
   \[  Q = \{x \in \frak{t}^* \,\, \vert\,\, b_i\leq \langle x, v_i\rangle \leq c_i,\,\,\forall\, 1 \leq i\leq k\},\]
   and define the \emph{symplectic cut of $M$ with respect to $Q$} to be the reduction of
$M_P$  (as given in the sentence below (\ref{cutting})) at $c$.

We note that if $M$ is in addition compact, then  both
$M_P$ and $M_Q$ are compact.

\subsection*{Duistermaat\--Heckman jump: concluding the proof of Theorem \ref{theorem1}}

\begin{proof}[Proof of  Theorem \ref{theorem1}]
We divide the proof into three steps.
Suppose that an $(n-2)$\--dimensional torus $T^{n-2}$ act on a
$2n$\--dimensional compact connected symplectic manifold $(M,\omega)$ with
fixed points, and that there is a commuting
$T^2$ symplectic action on $M$ which has a symplectic orbit.
\\
\\
{\bf Step 1}. By Theorem \ref{m} and Theorem \ref{ro}, we may assume that there is a generalized momentum map $\Phi: M\rightarrow T^{n-2}$. To prove that the action of $T^{n-2}$ is Hamiltonian, it suffices to prove that
for any splitting of $T^{n-2}$ of the form $T^{n-2}=S^1\times H,$ where $H$ is a $(n-3)$\--dimensional torus,
the action of $S^1$ is Hamiltonian.

Let
$\pi_1: T^{n-2}\rightarrow S^1$ and $\pi_2: T^{n-2} \rightarrow H$  be the projection maps. Then
 \[\phi:=\pi_1 \circ \Phi: M\rightarrow S^1\,\,\,\,\,\text{and}\,\,\,\,\,\phi_H:=\pi_2\circ \Phi: M\rightarrow H\] are generalized momentum maps with respect to the $S^1$\--action and  the
$H$\--action on $M$ respectively. Without loss of generality, we may assume that
the generalized momentum map $\phi$ has no local extreme value, cf. \cite[Lemma 2]{Mc1986}.

Let $h\in H$ be a regular value of $\phi_H$. Denote by
${\rm DH}_{T^{n-2}}: T^{n-2}\rightarrow \R$ the Duistermaat-Heckman function of the $T^{n-2}$ action on $M$. Define
\[ f: S^1 \rightarrow \R,\,\,\,\,\,\,\, f(x):={\rm DH}_{T^{n-2}}(x, h).\] To prove that the action of $S^1$ is Hamiltonian, it suffices to show that $f$ is strictly log-concave. This is because that if $f$ is strictly log-concave, then the generalized momentum map $\phi: M\rightarrow S^1$ cannot be surjective, and hence the action of $S^1$ on $M$ is Hamiltonian.

We will divide the rest of the proof into two steps, and we will use the following notation.  Consider the exponential map $\exp: \R\rightarrow S^1$, $t\mapsto {\rm e}^{{\rm i}2\pi t}$.
By abuse of notation, for any open connected interval $I=(b,c)\subset \R$ such that $\exp (I)\nsubseteq S^1$, we will identify $I$ with its image in $S^1$ under the exponential map.
 \\
 \\
{\bf Step 2}.  We first show that if $I_1=(b_1,c_1) \nsubseteq S^1$ is an open connected set consisting
of regular values of $\phi$, then $f$ is log\--concave on $I_1$.

  Indeed, suppose that for any $2\leq i\leq n-2$, $I_i=(b_i,c_i)\nsubseteq S^1$ is an open connected
  set such that
$$\mathcal{O}=I_2\times \cdots \times I_{n-2}\ni h$$ consists of regular values of $\phi_H$.
Then $I_1\times \mathcal{O} \nsubseteq T^{n-2}$ is a connected open set that consists of regular values of
$\Phi: M\rightarrow T^{n-2}$. Moreover, the action of $T^{n-2}$ on $$W=\Phi^{-1}(I\times \mathcal{O})$$ is Hamiltonian. Set
\[ Q= \{(x_1,\cdots, x_{n-2}) \in \R^{n-2}\,\,\, \vert\,\,\, b_i\leq x_i \leq c_i,\,\,\forall\, 1 \leq i\leq n-2\}.\]

The symplectic cutting  $W_Q$ of $W$ with respect to $Q$ is a compact Hamiltonian manifold from $W$, and,
on the interior of $Q$, the Duistermaat-Heckman function of $W_Q$ agrees with that of $W$.  Consequently, the
log\--concavity of the function $f$ follows from Theorem \ref{log-concavity-result}.
\\
\\
{\bf Step 3}. We show that if  $I_1=(b_1,c_1) \nsubseteq S^1$ is a connected open set which contains a unique critical value $a$ of $\phi$, then $f$ is strictly log-concave on $I_1$. We will use the same notation $W$ and $W_Q$ as in the Step 2. By definition, to show that
$f$ is strictly log-concave on $I_1$, it suffices to show that \[ \tilde{f}=f\circ \exp:  \R\rightarrow \R\]  is strictly log-concave on the pre-image of $I_1$ in $\R$ under the exponential map $\exp$. Again, by abuse of notation, we would not distinguish $I_1$ and its pre-image in $\R$, and we will use the same notation $a$ to denote the pre-image of $a \in I_1\nsubseteq S^1$ in $\exp^{-1}(I_1)$ .

To see the existence of such interval $I_1$, we need the assumption that the fixed point sets of $T^{n-2}$ is non-empty.
Since the fixed point submanifold of the action of $T^{n-2}$ on $M$  is a subset of $M^{S^1}$, it follows that $M^{S^1}$ is non-empty.
Let $X$ be a connected component of $M^{S^1}$. Choose $a\in S^1$ to be the image of $X$ under the generalized momentum map $\phi$.
Then a sufficiently small  connected open subset that contains $a$ will contain no other critical values of $\phi$.

Note that $X$ is a symplectic submanifold of $M$  which is invariant under the action of $H$, and that there are two symplectic manifolds associated to the point $(a,h)\in T^k=S^1\times H$: the symplectic quotient of $M$ (viewed as an $H$-space) at
$h$, and the symplectic quotient of $X$ (viewed as $T^k/S^1$-space) at $(a,h)$.  We will denote these two quotient spaces by $M_h$ and $X_h$ respectively. An immediate calculation shows that the dimension of  $M_h$ is six.

Since the action of $S^1$ commutes with that of $H$, by Lemma \ref{level-set} and Remark \ref{level-set-remark}, the action of $S^1$ preserves
the level set of the generalized momentum map $\phi_H$. Consequently, there is an induced action of $S^1$ on
$M_h$. Since $\phi: M\rightarrow S^1$ has a constant value on any $H$ orbit in $M$, it descends to a generalized momentum map $$\tilde{\phi}:M_h\rightarrow S^1.$$ It is easy to see that $X_h$ consists of the critical points of the generalized momentum map $\tilde{\phi}: M_h\rightarrow S^1$.  Moreover, since $\phi$ has no local extreme values in $M$, it follows that $\tilde{\phi}$ has no local extreme value in $M_h$. Thus for dimensional reasons, $X_h$ can only be of codimension four or six.

Now observe that the commuting symplectic action of $T^2$ on $M$ descends to a symplectic action on $M_h$ which commutes with the action of $S^1$
on $M^{S^1}$. It follows that $X_h$ is invariant under the induced action of $T^2$ on $M_h$ and so must contain a $T^2$ orbit. As a result, $X_h$ must have codimension four.

  Applying Theorem \ref{G-L-S88} to the compact Hamiltonian manifold $W_Q$, we get that

\[  \tilde{f}_+(a+t)-\tilde{f}_-(a-t)\,\,=\,\,\displaystyle\sum_{X_h}\dfrac{\text{vol\,}(X_h)}{(d-1)!\prod_j\alpha_j}(t-a)^{d-1}
+\mathcal{O}((t-a)^d).\]

In the above equation, $X_h$ runs over the collection of all
 connected component of $M^{S^1}$ that sits inside $\Phi^{-1}((a,h))$, the $\alpha_j$'s are the weights of the representation of $S^1$
on the normal bundle of $X_h$, and $d$ is half of the real dimension
 of $X_h$.

 By our previous work, $d=2$ for any  $X_h$; moreover, the two non-zero weights of $X_h$ must have opposite signs.
So the jump in the derivative is strictly negative, i.e.,
\[ \tilde{f}_+'(a)-\tilde{f}'_{-}(a)<0.\]
This completes the proof of Theorem \ref{theorem1}.
\end{proof}

\section{Examples and consequences of main results} \label{sec5}

\subsection*{Applications of Theorems \ref{theorem1} and \ref{log-concavity-result}}

Our proof of Theorem \ref{theorem1} depends on Lemma \ref{invariant-kahler-structure}. We note that the proof of Lemma \ref{invariant-kahler-structure} builds on the classification result of symplectic $T^2$ action on symplectic four manifolds, cf. \cite{P10} and \cite{DP10}. However, we have a simpler proof in the following special case, which was the goal of the article \cite{kim} (see Section 1 therein).

\begin{theorem}[Kim \cite{kim}]\label{complexity-one-mcduff} Let $(M,\, \omega)$ be a compact,
connected, symplectic $2n$\--dimensional manifold.
Then every effective symplectic action of an $(n-1)$\--dimensional torus $T^{n-1}$ on $(M,\, \omega)$
with non-empty fixed point set is Hamiltonian.
\end{theorem}

\begin{proof}
We prove the equivalent statement that if $(N,\, \omega)$ is $(2n-2)$\--dimensional compact, connected symplectic manifold, then every symplectic action of an $(n-2)$\--dimensional torus $T^{n-2}$ on $(N,\, \omega)$ with non-empty fixed point set is Hamiltonian.

 Indeed, suppose that there is a symplectic action of an $(n-2)$-dimensional torus on a $(2n-2)$\--dimensional compact connected symplectic manifold $N$.
 By Theorem \ref{m} and Theorem \ref{ro}, without  loss of generality we assume that there is a generalized momentum map $$\phi_N: N\rightarrow T^{n-2}.$$ Consider the action of $T^{n-2}$ on the product symplectic manifold $Z:=N\times T^2$
as given in equation (\ref{product-action}), where $T^2$ is equipped with the standard symplectic structure.  Define
\[ \Phi: N\times T^2\rightarrow T^{n-2},\,\,\,(x, t)\mapsto \phi_N(x), \,\,\textup{for all}\,\, x\in N,
\,\,\textup{for all}\,\, t\in T^2.\]
Then $\Phi$ is a generalized momentum map for the symplectic action of $T^{n-2}$ on $Z$. Moreover, there is an effective commuting
symplectic action of $T^2$ on $M$ with symplectic orbits given as in equation (\ref{product-action2}).

To show that the action of $T^{n-2}$ on $N$ is Hamiltonian, it suffices to show that the action of $T^{n-2}$ on $Z$ is Hamiltonian.
Using the same notation as in the proof of Theorem \ref{theorem1}, we explain that we can modify the argument there to prove that
the action of $T^{n-2}$ on $Z$ is Hamiltonian without using Lemma \ref{invariant-kahler-structure}.  We first note that the step 2 in the proof of Theorem \ref{theorem1} does not use
Lemma \ref{invariant-kahler-structure}, and that the same argument there applies to the present situation, which is a special case of what is considered in
Theorem \ref{theorem1}.

Lemma \ref{invariant-kahler-structure} is used in the step 1 in the proof of Theorem \ref{theorem1} to prove that $f$ is log-concave on $I_1$.
However, in the present situation, we can use Theorem \ref{complexity-one} instead to establish the log-concavity of the function $f$ on $I_1$. Since the proof of Theorem \ref{complexity-one} does not use Lemma \ref{invariant-kahler-structure}, Theorem \ref{complexity-one-mcduff}
can be established without using Lemma \ref{invariant-kahler-structure} as well.
\end{proof}

The following result  \cite[Proposition 2]{Mc1986} is a special case of Theorem \ref{complexity-one-mcduff}.

\begin{corollary}[McDuff \cite{Mc1986}] \label{Mc1986}
If $(M,\, \omega)$ is $4$\--dimensional compact, connected symplectic manifold,
then every effective symplectic $S^1$\--action on $(M,\, \omega)$ with non\--empty fixed point set is Hamiltonian.
\end{corollary}

Using \cite{Pr1973} we obtain the following result on measures.
A measure defined on the measurable subsets of $\mathbb{R}^n$ is
\emph{log\--concave} if for every pair
$A, B$ of convex subsets of $\mathbb{R}^n$ and for every $0<t<1$ such that
$tA + (1- t)B$ is measurable we have
$$
P(tA + (1 -t)B) \geq (P(A))^t(P(B))^{1-t}
$$
where the sign $+$ denotes Minkowski addition of sets.

\begin{theorem}[Pr\'ekopa \cite{Pr1973}]
If the density function of a measure $P$ defined on the measurable subsets of $\mathbb{R}^n$ is almost everywhere positive and is a log\--concave function, then the measure $P$ itself is  log\--concave.
\end{theorem}

These notions extend immediately to our setting where $\mathbb{R}^n$ is replaced
by the dual Lie algebra $\mathfrak{t}^*$, equipped with the push\--forward
of the Liouville measure on $M$, and we have, as a consequence of Theorem
\ref{log-concavity-result}, the following.

\begin{theorem} \label{measures}
 Let $T$ be an $(n-2)$\--torus which acts effectively on a
compact connected symplectic $2n$\--manifold $(M,\omega)$ in a Hamiltonian fashion. Suppose that there is a commuting
symplectic action of a $2$\--torus $S$ on $M$ whose orbits are symplectic.
Then the Duistermaat-Heckman measure is log-concave.
\end{theorem}

\subsection*{Examples}

Next we give some more examples to which our theorems apply.

\begin{example} \label{bb}
The Kodaira\--Thurston
     manifold is the symplectic manifold  \[{\rm KT}:=\left(\mathbb{R}^2 \times T^2/\mathbb{Z}^2\right).\]
   This is the case in Theorem \ref{theorem1} which is ``trivial" since $n-2=0$,
     and hence $T^{n-2}=\{e\}$ (in fact this manifold is non\--Kahler, and as such it does not
     admit any Hamiltonian torus action, of any non\--trivial dimension).

     Nevertheless   ${\rm KT}$ serves to illustrate a straightforward case
     which admits the  symplectic transversal symmetry, and it may be used to construct
     lots of examples satisfying the assumptions of the theorem, eg. ${\rm KT} \times S^2$.
          \end{example}

\begin{example} \label{xxxxx}
This example is a generalization of Example \ref{bb}.
Theorem \ref{theorem1} covers an infinite class of symplectic manifolds with
Hamiltonian $S^1$\--actions, for instance
\begin{eqnarray}
M:=\left((\mathbb{CP}^1)^2 \times \mathbb{T}^2 \right)/ \mathbb{Z}_2, \label{abc}
\end{eqnarray}
where $S^1$ acts Hamiltonianly on the left factor $(\mathbb{CP}^1)^2$, and $\T^2$ acts symplectically with
symplectic orbits on the right factor; the quotient is taken with respect to the natural diagonal action of
$\mathbb{Z}_2$ (by the antipodal action on a circle of $\mathbb{T}^2$, and by rotation by 180 degrees
about the vertical axes of the the spheres $S^2$), and the action on
$(\mathbb{CP}^1)^2 \times \mathbb{T}^2$ descend to an action on $M$. Any product symplectic
form upstairs also descends to the quotient.
\end{example}

\begin{example}
The $2n$\--dimensional
symplectic manifolds with $\T^n$\--actions in Theorem \ref{theorem1} are a subclass of those
classified in \cite{P10}, which, given that $\T^2$ is $2$\--dimensional,
are of the form $M_{Z}$ described in (\ref{efg}).

 Let $Z$ be any good $(2n-2)$\--dimensional orbifold, and let $\widetilde{\Sigma}$
be its orbifold universal cover. Equip $Z$ with any orbifold symplectic form, and $\widetilde{\Sigma}$
with the induced symplectic form. Equip $\T^2$ with any area form.
Equip $\widetilde{Z} \times \T^2$ with the product symplectic form, and the symplectic $\T^2$\--action
with symplectic orbits by translations on the right most factor. This symplectic $\T^2$\--action descends
to a symplectic $\T^2$\--action with symplectic orbits on
\begin{eqnarray}
M_Z:=\widetilde{Z} \times_{\pi_1^{\textup{orb}}(Z)}\T^2 \label{efg}
\end{eqnarray}
considered with the induced product symplectic form.

 An a priori intractable question is: \emph{give an
explicit list of  manifolds $M_{Z}$ which admit a Hamiltonian $S^1$\--action}. On the other hand,
examples of such $M_Z$, and hence fitting in the statement of Theorem \ref{theorem1},
may be constructed in all dimensions, eg. the manifold in (\ref{abc}).

Theorem \ref{theorem1} says that any commuting symplectic $\T^{n-2}$ action on (\ref{efg}) \emph{which has
fixed points}, must be Hamiltonian.
\end{example}

\section{Final Remarks} \label{remarks}

\subsection*{The theorems of Frankel and McDuff} \label{FrMc}
The relation between conditions (i) and (ii) in Section 1
can be traced back to a result proved
by T. Frankel in 1959, which says that if $(M,\, \omega)$ is $2n$\--dimensional
compact K\"ahler manifold, then every symplectic $S^1$\--action on $(M,\, \omega)$ with
fixed points is Hamiltonian \cite{Fr1959}. In 1986 D. McDuff showed that if $(M,\, \omega)$ is $4$\--dimensional symplectic manifold, then every symplectic $S^1$\--action on $(M,\, \omega)$ with
fixed points is Hamiltonian \cite[Proposition 2]{Mc1986}.

In the same paper, McDuff  gave an example of a symplectic
$S^1$\--action on a compact connected symplectic six manifold $(M,\, \omega)$ which has non\--empty fixed
point set but is not Hamiltonian. However, the fixed point set in her example is a submanifold,
and to this day there are no known examples of non\--Hamiltonian symplectic $S^1$\--actions on
$6$\--manifolds (or manifolds of any dimension $2n\ge 6$) which have a discrete fixed point set. This has given rise to what some authors have called the McDuff\--Frankel question: is there such an example, or,  on the contrary, is every
such action Hamiltonian?

A large number of papers have been written on the subject, but to this day
the answer is, as far as we know, unknown. The answer is known to be yes for semifree actions, see  Tolman\--Weitsman \cite[Theorem 1]{ToWe2000}.

Feldman \cite[Theorem 1]{Feldman2001}
characterized the obstruction for
a symplectic circle action on a compact manifold to be
Hamiltonian and deduced the McDuff and Tolman-Weitsman theorems by applying his
criterion. He showed that the Todd genus of a manifold admitting a symplectic circle action with isolated
fixed points is equal either to $0$, in which case the action is non-Hamiltonian, or to $1$, in which
case the action is Hamiltonian. In addition, any symplectic circle action on a compact manifold with positive Todd genus is Hamiltonian. For additional results regarding aspherical
symplectic manifold see \cite[Section 8]{KeRuTr2008}
and \cite{LuOp1995}.

Ono proved that if  $(M, \omega)$ is a compact connected symplectic $S^1$\--manifold which is
 a Lefschetz manifold\footnote{Let $(M ,\omega)$ be a symplectic $2n$-manifold and, for each $k=0,\ldots,n$, consider the linear map
$${\rm H}^{n-k}(M, \mathbb{R})\ni [\alpha]
\mapsto [\alpha\wedge \omega^k]\in {\rm H}^{n+k}(M,\mathbb{R}).$$ If the map
for $k=n-1$ is an isomorphism, then $(M, \omega)$ is called a \textit{Lefschetz manifold}. If all these maps for $k =0, \ldots, n$
are isomorphisms, then $(M, \omega)$ is called a \textit{strong
Lefschetz manifold}. }, then
if the $S^1$-action has fixed points, then it is
Hamiltonian.  Y. Lin proved recently \cite[Theorem 4.5]{Lin2007} that
there exist infinitely many topologically inequivalent $6$\--dimensional compact Lefschetz manifolds with
Hamiltonian $S^1$\--actions with fixed points, and which are
not homotopy equivalent to any compact K\"ahler manifold. Hence
within the category of compact connected symplectic manifolds we have an inclusion
$$
{\small \{\textup{$S^1$\--K\"ahler manifolds with fixed points}\} \subsetneq \{\textup{$S^1$\--Lefschetz manifolds with fixed points}\} }
$$
and hence Ono's Theorem covers a class of manifolds which is strictly larger than that covered in
 Frankel's theorem.

More recently,  Cho, Hwang, and Suh \cite{CHS10} studied symplectic circle actions on compact connected
$6$\--manifolds $(M,\omega)$ with non-isolated fixed points.
Assume that the the cohomology class $[\omega]$ is integral and that there exists a circle valued momentum map
$\mu: M\rightarrow S^1.$
These authors observed that if the Duistermaat-Heckman function of the circle valued momentum map
$F_{\mu}:S^1\rightarrow \R$  is log-concave on
$S^1$, then it must be a constant. However, if the fixed point set is non-empty, the Duistermaat-Heckman function can not be a constant since there will be a non-zero jump in its value moving across a critical level set.  In particular, combining this observation with a result in Lin \cite{Lin07} concerning the log-concavity of the Duistermaat-Heckman function, they proved that if the regular symplectic quotient taken with respect to the circle valued momentum map satisfying $b_+=1$, then the symplectic circle action must be Hamiltonian if its fixed point set is non-empty.

Finally, in the context of orbifolds, an analogue of Tolman and Weitsman's result was obtained by Godinho \cite{Go2006}.

\subsection*{Results on higher dimensional groups}
For higher dimensional Lie groups, some results are known, see
for instance Giacobbe \cite[Theorem 3.13]{Giacobbe2005}, Duistermaat\--Pelayo
\cite[Corollary 3.9]{DP10} and Ginzburg \cite[Proposition 4.2]{Gi1992}.
A Hodge theoretic approach for non\--compact manifolds was introduced by Pelayo\--Ratiu recently \cite[Theorem 1]{PeRa2010}.

\subsection*{Minimal fixed point set of symplectic actions}
The McDuff\--Frankel question is related to the question of
what is the minimal number of fixed points of a symplectic circle action
that has non\--empty fixed point set. If the action is Hamiltonian, this
is easily shown to be $n+1$ using Morse theory for the momentum map
of the action.

Motivated by the work of Tolman and Weitsman on semifree actions, in which they use
equivariant cohomological methods to approach this question,
Pelayo\--Tolman used a technique  \cite[Lemma 8]{PeTo2011} involving Vandermonde determinants. Using
 this technique they showed that the number of fixed points of a symplectic circle action with
non\--empty fixed point set on a compact symplectic $2n$\--dimensional manifold is at least $n+1$
(see \cite[Theorem 1]{PeTo2011}), provided the Chern class map
is somewhere injective (the \emph{Chern class map} $M^{S^1} \to \mathbb{Z}$
assigns to a fixed point the sum of the weights of the action at that point).
Recently, several authors have continued this equivariant cohomological approach, proving a number of
interesting results, see for instance Li \cite[Theorem 1.9]{Li2010}, Li\--Liu \cite{LiLi2010, LiLi2011} and the
references therein.

\section{Appendix:  actions with symplectic orbits} \label{appendix}

This section is a review of \cite[Sections 2, 3]{P10}, with a new remark at the end.
The remark is that many of the results
therein extends to the case of
symplectic torus actions with symplectic orbits
on non\--necessarily compact symplectic manifolds $(X,\,\omega)$.  We needed this
remark to prove Theorem \ref{log-concavity-result}.

Let $(X, \,\sigma)$ be a connected symplectic
manifold equipped with an effective symplectic action of a torus $T$ for which there is at least
one  $T$\--orbit which is a $\op{dim}T$\--dimensional symplectic submanifold of $(X, \, \sigma)$
(this implies that all are). \emph{We do not assume that $X$ is necessarily compact.}
Then there exists a unique non\--degenerate antisymmetric bilinear form
$\sigma^{\mathfrak{t}} \colon \mathfrak{t} \times \mathfrak{t} \to \R$
on the Lie algebra $\mathfrak{t}$ of $T$
such that
$
\sigma_x(u_X(x), \,v_X(x))=\sigma^{\mathfrak{t}}(u,v),
$
for every $u, \,v \in \mathfrak{t}$, and every $x \in X$.

One can check that the stabilizer subgroup $T_x$ of the $T$\--action at every point  $x \in X$ is
a finite group.

\subsection*{Orbit space $X/T$}

As usual, $X/T$  denotes the orbit space of the $T$\--action.
Let $\pi \colon X \to X/T$ the canonical projection. The space $X/T$
is provided with the maximal topology for which $\pi$ is continuous; this topology is Hausdorff.
 Because $X$ is connected, $X/T$ is
 connected.

 \subsection*{Orbifold atlas on $X/T$}
Let $k:=\op{dim}X-\op{dim}T$.
By the tube theorem of Koszul (see eg. \cite[Theorem \,B24]{GGK}),
for each $x \in X$ there exists a $T$\--invariant open neighborhood $U_x$ of the $T$\--orbit
 $T \cdot x$ and a $T$\--equivariant diffeomorphism $\Phi_x$ from $U_x$ onto the
 associated bundle $T \times_{T_x} D_x$, where $D_x$ is an open disk centered at the origin in
 $\R^k \cong \C^{k/2}$
 and $T_x$ acts by linear transformations on $D_x$.

The action of $T$ on $T \times_{T_x} D_x$ is induced by the action of $T$ by translations on
the left factor of $T \times D_x$.
 Because $\Phi_x$
is a $T$\--equivariant diffeomorphism, it induces a homeomorphism $\widehat{\Phi}_x$
 on the quotient $\widehat{\Phi}_x \colon D_x/T_x \to \pi(U_x)$, and there
 is a commutative diagram of the form
$$
\xymatrix{
T \times D_x \ar[r]^{\pi_x}      &  T  \times_{T_x} D_x \ar[d]^{p_x}  \ar[r]^{\Phi_x}  & U_x  \ar[d]^{\pi|_{U_x}} \\
 D_x \ar[u]^{i_x} \ar[r]^{\pi'_x} &  D_x/T_x   \ar[r]^{\widehat{\Phi}_x} &      \pi(U_x)}
$$
where $\pi_x$, $\pi'_x$, $p_x$ are the canonical projection maps, and $i_x(d):=({\rm e},\,d)$ is the inclusion map.
Let
$
\phi_x:= \widehat{\Phi}_x \circ \pi'_x.
$
The collection of charts
$$
\widehat{\mathcal{A}}:=\{(\pi(U_x), \, D_x, \, \phi_x, \, T_x)\}_{x \in X}
$$
is an orbifold atlas for  $X/T$. We call $\mathcal{A}$ the class of atlases
 equivalent to the orbifold atlas  $\widehat{\mathcal{A}}$.
 We denote the orbifold $X/T$ endowed with
the class $\mathcal{A}$
by $X/T$, and the class $\mathcal{A}$ is assumed.

\subsection*{Flat connection}
The collection $\Omega=\{\Omega_x\}_{x \in X}$ of subspaces
$\Omega_x \subset \op{T}_xX$,
where $\Omega_x$ is the $\sigma_x$\--orthogonal complement to $\op{T}_x(T \cdot x)$ in
$\op{T}_xX$,
for every $x \in X$, is a smooth distribution on $X$. The projection
mapping $\pi \colon X \to X/T$ is a smooth principal $T$\--orbibundle
for which $\Omega$ is a $T$\--invariant flat connection. Let $\mathcal{I}_x$ be the maximal integral manifold of the
distribution $\Omega$.

The inclusion $i_x \colon \mathcal{I}_x \to X$ is an injective immersion
between smooth manifolds and
the composite $\pi \circ i_x \colon \mathcal{I}_x \to X/T$ is an
orbifold covering map. Moreover, there exists a unique $2$\--form $\nu$ on $X/T$ such that
$\pi^* \nu|_{\Omega_x}=\sigma|_{\Omega_x}$ for every $x \in X$.
The form $\nu$ is symplectic, and so the pair $(X/T,\, \nu)$
is a connected symplectic orbifold.

\subsection*{Model for $X/T$}
We define the space that we call
the \emph{$T$\--equivariant symplectic model
$(X_{\textup{model},\, p_0},\, \sigma_{\textup{model}})$
of $(X, \, \sigma)$ based at a regular point $p_0 \in X/T$} as follows.
\begin{itemize}
\item[i)]
The space $X_{\textup{model},\, p_0}$ is the associated bundle
$$
X_{\textup{model}, \, p_0}:=\widetilde{X/T} \times_{\pi_1^{\textup{orb}}(X/T, \, p_0)} T,
$$
where the space $\widetilde{X/T}$ denotes the orbifold universal
cover of the orbifold $X/T$
based at a regular point $p_0 \in X/T$, and the orbifold fundamental group
$\pi^{\textup{orb}}_1(X/T,\,p_0)$ acts on the Cartesian product $\widetilde{X/T}
\times T$ by the diagonal action $x \, (y,\,t)=(x \star y^{-1},\,\mu(x) \cdot t)$, where
$\star \colon \pi_1^{\textup{orb}}(X/T, \, p_0) \times \widetilde{X/T} \to \widetilde{X/T}$ denotes the natural
action of $\pi_1^{\textup{orb}}(X/T, \, p_0)$ on $\widetilde{X/T}$,
and $\mu \colon \pi_1^{\textup{orb}}(X/T, \, p_0)
\to T$ denotes the monodromy homomorphism of
$\Omega$.
\item[ii)]
The symplectic form  $\sigma_{\textup{model}}$ is induced on the quotient by the product
symplectic form on the Cartesian product $\widetilde{X/T} \times T$.
The symplectic form on $\widetilde{X/T}$ is defined as the pullback by the
orbifold universal covering map $\widetilde{X/T} \to X/T$ of $\nu$.
The symplectic form on the torus
$T$ is the unique $T$\--invariant symplectic form $\sigma^T$ determined by
$\sigma^{\mathfrak{t}}$.
\item[iii)]
The action of $T$ on the space $X_{\textup{model},\, p_0}$  is the action of $T$ by translations
which descends from the action of $T$ by translations on the right factor of the product
$\widetilde{X/T} \times T$.
\end{itemize}
In this definition we are implicitly using that the universal
cover $\widetilde{X/T}$ is a smooth manifold and the orbit space $X/T$ is a good orbifold,
which is proven by analogy with \cite{P10}.

\subsection*{Model of $(X,\sigma)$ with $T$\--action}
In \cite{P10}, the following
theorem was shown (see Theorem 3.4.3).

\begin{theorem}[Pelayo \cite{P10}] \label{tt}
Let $(X, \,\sigma)$ be a compact connected  symplectic manifold
equipped
with an effective symplectic action of a torus $T$, for which at least one, and hence every $T$\--orbit is
a $\op{dim}T$\--dimensional symplectic submanifold of $(X,\, \sigma)$. Then
$(X,\,\sigma)$ is $T$\--equivariantly
symplectomorphic to its $T$\--equivariant symplectic model based
at any regular point $p_0 \in X/T$.
\end{theorem}

The purpose of this appendix was to point out that Theorem \ref{tt}
with the word ``compact" removed from the statement  holds. We
refer to \cite{P10} for a proof in the case that $X$ is
compact, which works verbatim in the case that $X$ is not compact.
\\
\\
{\bf Acknowledgements.}   AP thanks Helmut Hofer for the hospitality at IAS and for discussions.

AP was partly supported by an NSF Postdoctoral Fellowship, NSF Grants
DMS-0965738 and DMS-0635607, an NSF CAREER Award, an Oberwolfach Leibniz Fellowship,
Spanish Ministry of Science Grant MTM 2010-21186-C02-01,
and by the Spanish National Research Council.

YL first learned the relationship between the log-concavity
conjecture and the McDuff-Frankel question from Yael Karshon. He would like to thank her for many stimulating
mathematical discussions while he was a postdoctoral fellow at the University of Toronto, and for her continuous support.

\noindent
\\
Yi Lin \\
Department of Mathematical Sciences \\
Georgia Southern University  \\
203 Georgia Ave., Statesboro, GA, USA  \\
{\em Email}: yilin@georgiasouthern.edu
\\
\\
\\
{\'A}lvaro Pelayo \\
School of Mathematics\\
Institute for Advanced Study\\
Einstein Drive, Princeton, NJ 08540 USA.\\
 \\
 and
 \\
 \\
\noindent
Washington University,  Mathematics Department \\
One Brookings Drive, Campus Box 1146\\
St Louis, MO 63130-4899, USA.\\
{\em Website}: \url{http://www.math.wustl.edu/~apelayo/}\\
{\em E\--mail}: {apelayo@math.wustl.edu} and {apelayo@math.ias.edu}

\end{document}